\documentclass{amsart}
\usepackage[utf8]{inputenc}
\usepackage[T1]{fontenc}
\usepackage{lmodern}
\usepackage{enumerate}
\usepackage{amssymb,amsmath,amscd,amsthm,mathtools,nicefrac,textcomp,microtype}
\mathtoolsset{mathic}
\usepackage{etex,tikz,tikz-cd}
\usepackage{graphicx,color}
\usepackage{booktabs,array}

\usepackage{paralist}
\usepackage[font=small,margin=2em]{caption}
\usepackage{chngpage}
\usepackage[all]{xy}
\usepackage{url}
\usepackage[draft=false,breaklinks=true,hidelinks=true]{hyperref}  
\usepackage[nameinlink,capitalise,noabbrev]{cleveref} 

\newcommand{\RR}{\mathbb{R}}
\newcommand{\QQ}{\mathbb{Q}}
\newcommand{\CC}{\mathbb{C}}

\newcommand{\ZZ}{\mathbb{Z}}
\newcommand{\NN}{\mathbb{N}}

\newcommand{\SU}{\operatorname{SU}}

\newcommand{\rKO}{\reduced{\mathrm{KO}}}
\newcommand{\rKSpin}{\reduced{\mathrm{KSpin}}}
\newcommand{\reduced}[1]{\smash{\widetilde{#1}}}

\newcommand{\Rep}{\mathrm{Rep}}
\newcommand{\Vect}{\mathrm{Vect}}

\DeclareMathOperator{\diag}{diag}

\newcommand{\abs}[1]{|#1|}

\renewcommand{\mod}{\text{ mod }}

\newtheorem{theorem}{Theorem}                   
\newtheorem{corollary}[theorem]{Corollary}      
\newtheorem{proposition}[theorem]{Proposition}  
               \crefname{maintheorem}{Theorem}{Theorems}

\newtheorem{innercustomgeneric}{\customgenericname}\crefname{innercustomgeneric}{Theorem}{Theorems}
\providecommand{\customgenericname}{}
\newcommand{\newcustomtheorem}[2]{%
  \newenvironment{#1}[1]
  {%
   \renewcommand{\customgenericname}{#2}%
   \renewcommand{\theinnercustomgeneric}{##1}%
   \innercustomgeneric
  }
  {\endinnercustomgeneric}
}
\newcustomtheorem{customtheorem}{Theorem}

\theoremstyle{definition}
\newtheorem{definition}[theorem]{Definition}    

\theoremstyle{remark}
\newtheorem{remark}[theorem]{Remark}            
\newtheorem{example}[theorem]{Example}          
\newtheorem{examples}[theorem]{Examples}        
\newtheorem{statistics}[theorem]{Statistics}    

\makeatletter
\newcommand\fg{f.\,g.\@ifnextchar){}{\ }} 
\newcommand\fd{f.\,d.\@ifnextchar){}{\ }} 
\makeatother
\newcommand{\ie}{i.\,e.\ }

\title[Open manifolds with positively curved souls]{Open manifolds with non-homeomorphic positively curved souls}
\date{\today}

\author[D.~González-Álvaro]{David González-Álvaro}
\address{Université de Fribourg, Switzerland}
\thanks{D. G.-\'A. received support from SNF grant 200021E-172469, the DFG Priority Programme SPP 2026 (Geometry at Infinity), and from MINECO grants MTM2014-57769-3-P and MTM2014-57309-REDT}
\email{david.gonzalezalvaro@unifr.ch}

\author[M.~Zibrowius]{Marcus Zibrowius}
\address{Heinrich-Heine-Universität Düsseldorf, Germany}
\thanks{M. Z. was partially supported by the DFG Research Training Group GRK~2240 (Algebro-geometric Methods in Algebra, Arithmetic and Topology)}
\email{marcus.zibrowius@cantab.net}


\begin{document}

\begin{abstract}
  We extend two known existence results to simply connected manifolds with positive sectional curvature: we show that there exist pairs of simply connected positively-curved manifolds that are tangentially homotopy equivalent but not homeomorphic, and we deduce that an open manifold may admit a pair of non-homeomorphic simply connected and positively-curved souls.  Examples of such pairs are given by explicit pairs of Eschenburg spaces.  To deduce the second statement from the first, we extend our earlier work on the stable converse soul question and show that it has a positive answer for a class of spaces that includes all Eschenburg spaces.
\end{abstract}

\maketitle

\section{Introduction}
The Soul Theorem \cite{CG} determines the structure of an open manifold $N$ endowed with a metric $g$ of non-negative sectional curvature: there exists a closed totally convex submanifold $S$, called the soul, such that $N$ is diffeomorphic to the normal bundle of $S$. This soul may not be unique, but for a given metric \(g\) any two souls are isometric.  Our work is motivated then by the following question: if $N$ admits different non-negatively curved metrics $g_1,g_2$, what can be said about the corresponding souls $S_1,S_2$? For convenience we will say that \(S\) is a soul of \(N\) iff \(S\) is a soul of \((N,g)\) in the usual sense for \emph{some} metric \(g\) of non-negative sectional curvature.

Open manifolds with \emph{different} souls can be constructed in the following ways. It is well known that there exist $3$-dimensional lens spaces $L_1,L_2$ that are homotopy equivalent but not homeomorphic, and such that their products with $\RR^3$ are diffeomorphic \cite[\S\,2]{Mil61}.
Thus, the obvious product metrics on $L_1\times\RR^3\cong L_2\times\RR^3$  have two non-homeomorphic souls.
In a similar vein, all of the fourteen exotic $7$-dimensional spheres $\Sigma^7$ (\ie manifolds which are homeomorphic but not diffeomorphic to the standard sphere $\mathbb{S}^7$) admit non-negatively curved metrics (see \cite{GZ00} and the recent preprint \cite{GKS}), and they all become diffeomorphic after taking the product with $\RR^3$. Thus, the obvious product metrics yield fifteen non-diffeomorphic souls of $\mathbb{S}^7\times\RR^3$.

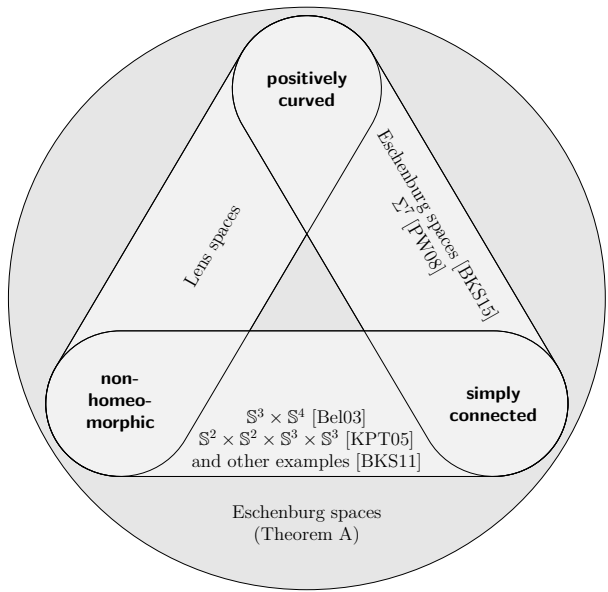
\begin{figure}
  \resizebox{8cm}{8cm}{\begin{tikzpicture}[font=\Large]
  \node at (0,0) (pc) {};
  \node at (-5cm,-8.66cm) (nh) {};
  \node at (5cm,-8.66cm) (sc) {};
  \node at (0,-5.77cm) (center) {};

  \node at (-2.5cm,-4.33cm) (pc-nh) {};
  \node at (2.5cm,-4.33cm) (pc-sc) {};
  \node at (0,-8.66cm) (nh-sc) {};

  \filldraw[fill=white!90!black] (center) circle (8cm);

  \tikzstyle{area} = [anchor=center,shape=rectangle,
                      minimum width=14cm,
                      minimum height=4cm,
                      rounded corners=2cm,
                      fill=white!95!black,draw=black,text opacity=1,
                      align=center,inner sep=0,font=\LARGE]
  \node[area,rotate around={ 60:(pc-nh)}] at (pc-nh)
     {Lens spaces};
  \node[area,rotate around={-60:(pc-sc)}] at (pc-sc)
     {Eschenburg spaces \cite{BKS15}\\\(\Sigma^7\) \cite{PW}\\~\\~\\~};
  \node[area,rotate around={  0:(nh-sc)}] at (nh-sc)
     {~\\~\\~\\~\\\(\mathbb{S}^3\times\mathbb{S}^4\) \cite{Bel}\\ \(\mathbb{S}^2\times\mathbb{S}^2\times\mathbb{S}^3\times\mathbb{S}^3\) \cite{KPT}\\and other examples \cite{BKS11}};
  \node[area,fill opacity=0,rotate around={ 60:(pc-nh)},draw=black] at (pc-nh) {};
  \node[area,fill opacity=0,rotate around={-60:(pc-sc)},draw=black] at (pc-sc) {};
  \node[area,fill opacity=0,rotate around={  0:(nh-sc)},draw=black] at (nh-sc) {};
  \node[align=center,font=\LARGE] at (0,-12) {Eschenburg spaces\\(\Cref{THM: nondiffeo souls with positive curv})};
  \tikzstyle{corner}=[font=\bf\sffamily\LARGE,align=center,text=black]
  \node[corner] at (pc) {\bf\sffamily positively\\\bf\sffamily curved};
  \node[corner] at (nh) {non-\\homeo-\\morphic};
  \node[corner] at (sc) {simply\\connected};
\end{tikzpicture}

}
  \caption{Existing results on pairs of distinct souls}\label{fig:main-result}
\end{figure}

In a more elaborate construction, Belegradek showed that $\mathbb{S}^3\times\mathbb{S}^4\times\RR^5$ admits infinitely many souls that are pairwise non-homeomorphic \cite{Bel}. In \cite{KPT} the same statement was shown over $\mathbb{S}^2\times\mathbb{S}^2\times\mathbb{S}^3\times\mathbb{S}^3\times\RR^k$ for any $k>10$, where the souls satisfy certain curvature-diameter properties. Finally, in  \cite{BKS11} further examples in the same vein were constructed with the additional property that the souls have codimension four.

Our main interest in this note is the existence of souls with \emph{positive} sectional curvature.  For example, the lens spaces described above have metrics with constant positive sectional curvature. Unpublished work by Petersen-Wilhelm \cite{PW} announces a positively curved metric on one of the exotic spheres $\Sigma^7$; this would yield two non-diffeomorphic souls with positive curvature on $\mathbb{S}^7\times\RR^3$.
It also follows from \cite{BKS15} that there exist open manifolds with pairs of non-diffeomorphic homeomorphic souls with positive curvature:
see \cref{THM: BKS15 nontrivial line bundles} below for the precise statement and its proof.
In all of the above examples, however, the pairs of souls satisfy at most two of the following three properties: they are simply connected, they are non-homeomorphic, they have positive sectional curvature. The situation is summarized in \cref{fig:main-result}. Here, we present open manifolds with pairs of souls that satisfy all three properties simultaneously:

\begin{customtheorem}{A}\label{THM: nondiffeo souls with positive curv}
There exist simply connected open manifolds with a pair of non-homeomorphic souls of positive sectional curvature.
\end{customtheorem}

In combination with results of \cite{KPT,BKS11}, \cref{THM: nondiffeo souls with positive curv}   yields some consequences on the topology of the moduli space of Riemannian metrics with non-negative sectional curvature on the corresponding spaces. This is explained in Section~\ref{sec: mmoduli spaces}.

\Cref{THM: nondiffeo souls with positive curv} will be proved in the following more explicit form:

\begin{customtheorem}{\ref{THM: nondiffeo souls with positive curv}\textquotesingle}\label{THM: construction for THM A}
There exist Eschenburg spaces $M$ with the following property:  the total space of every real vector bundle over $M$ of rank \(\geq 8\) admits a pair of non-homeomorphic souls of positive sectional curvature.
\end{customtheorem}

Of course, one of the souls is the given Eschenburg space \(M\); the other soul is a homotopy equivalent but non-homeomorphic Eschenburg space \(M'\).  Recall that Eschenburg spaces \cite{Esch82} form an infinite family of $7$-dimensional quotients of $SU(3)$ under certain circle actions.  They inherit non-negatively curved metrics from \(SU(3)\) which in many cases have positive sectional curvature (see \cref{sec: Esch} for details).  The only known examples of pairs of simply connected manifolds with positive curvature which are homotopy equivalent but non-homeomorphic occur among these Eschenburg spaces \cite{CEZ07,Sha02}.  On the other hand, there are only finitely many homeomorphism classes of Eschenburg spaces in each homotopy type \cite[Prop.~1.7]{CEZ07}, so our strategy behind proving \cref{THM: construction for THM A} cannot yield infinite families of non-homeomorphic souls.%
\medskip

This strategy is as follows.  We use the classical fact that \emph{the total spaces of a vector bundle of high rank and its pull-back under a tangential homotopy equivalence are diffeomorphic}.
Here, two manifolds \(M_1,M_2\) of the same dimension are called \emph{tangentially homotopy equivalent} if there exists a homotopy equivalence \(f\colon M_1\to M_2\) such that the tangent bundle \(TM_1\) and \(f^*TM_2\) are stably isomorphic, \ie such that \(TM_1\times\RR^k\) and \(f^*TM_2\times\RR^k\) are isomorphic as bundles over \(M_1\) for some integer \(k\geq 0\). Thus, \cref{THM: construction for THM A} is a consequence of the two following results, in which each Eschenburg space is understood to come equipped with some metric which descends from a circle invariant non-negatively curved metric on \(SU(3)\).

\begin{customtheorem}{B}\label{THM: existence of tangential nonhomeo pairs}
There exist pairs of positively curved Eschenburg spaces which are tangentially homotopy equivalent but not homeomorphic.
\end{customtheorem}

\begin{customtheorem}{C}\label{THM: bundles over Esch}
Let \(M\) be an Eschenburg space. The total space of every real vector bundle over $M$ of rank $\geq 8$ admits a metric with non-negative sectional curvature whose soul is isometric to $M$.
\end{customtheorem}

Explicit pairs of Eschenburg spaces as in \cref{THM: existence of tangential nonhomeo pairs} are listed in \cref{table:Esch:phenomena} below.  They constitute the first known examples of simply connected positively curved non-homeomorphic spaces that are tangentially homotopy equivalent.
On the other hand, any two homeomorphic Eschenburg spaces are in particular tangentially homotopy equivalent.  (This implication holds for many closed manifolds of dimension at most $7$; see \cref{rem: homeo type s mfds are tangentially homotopic}.)
Pairs of simply connected \emph{non-negatively} curved manifolds that are tangentially homotopy equivalent but not homeomorphic are already known:  Crowley exhibited an explicit such pair of $\mathbb{S}^3$-bundles over $\mathbb{S}^4$ \cite[p.\,114]{CPhD}, which carry metrics of non-negative sectional curvature by the work of Grove and Ziller \cite{GZ00}.

\medskip

\Cref{THM: bundles over Esch} should be seen in the context of the \emph{converse soul question}: does every vector bundle over a manifold with non-negative sectional curvature itself admit a metric of non-negative sectional curvature?  While this is known to be false for general base manifolds, very little is known about this question for simply connected bases.  Every vector bundle over a sphere $\mathbb{S}^n$ with $2\leq n\leq 5$ admits such a metric \cite{GZ00}, and there exist partial positive results over cohomogeneity-one four-manifolds \cite{GZ11}.
A stable version of the question is known to have an affirmative answer for all spheres \cite{Ri}, and also for many other families of homogeneous spaces including almost all the positively curved ones \cite{Go17,GoZ}.  On the other hand, there is not a single known example of a vector bundle over a simply connected non-negatively curved closed manifold whose total space admits no metric of non-negative sectional curvature.  Using the same techniques as in the proof of \cref{THM: bundles over Esch}, we can further extend the list of examples in which the converse soul question has a positive answer, at least after some form of stabilization:

\begin{customtheorem}{\ref{THM: bundles over Esch}\textquotesingle}\label{THM: bundles over other spaces}  Let $M$ be any of the closed manifolds listed below.  The total space of every real vector bundle over $M$ of rank $\geq r$ admits a metric of non-negative sectional curvature, where \(r\) depends on $M$ as listed:
\begin{compactitem}
\item Generalized Witten spaces \(M\) with \(H^4(M)\) of odd order (\(r = 8\)).
\item Generalized Witten spaces \(M\) with \(H^4(M)\) of even order (\(r = 18\)).
\item Products of spheres $\mathbb{S}^2\times \mathbb{S}^{m}$ with $m\equiv 3,5\mod 8$ (\(r = m + 3\)).
\item The total space of the unique non-trivial linear $\mathbb{S}^m$-bundle over $\mathbb{S}^2$ where either $m=3$ or $m\equiv 5\mod 8$ (in any case \(r = m + 3\)).
\end{compactitem}
\end{customtheorem}

The generalized Witten spaces appearing here are a family of manifolds \(M_{k,l_1,l_2}\) defined as quotients of \(\mathbb{S}^5\times \mathbb{S}^3\) under the circle action
\[
\begin{array}{ccc}
S^1 \times \mathbb{S}^5\times \mathbb{S}^3 & \rightarrow & \mathbb{S}^5\times \mathbb{S}^3  \\
\left(z ,( u_1,u_2,u_3), (v_1,v_2)\right) & \mapsto & \left(( z^k u_1,z^k u_2,z^k u_3), (z^{l_1} v_1,z^{l_2} v_2)\right)
\end{array}
\]
where \(\mathbb{S}^5\subset\CC^3\), \(\mathbb{S}^3\subset\CC^2\), and \(k,l_1,l_2\) are nonzero integers such that \(k,l_j\) are coprime for \(j=1,2\); for such a space \(H^4(M_{k,l_1,l_2})=\ZZ_{l_1 l_2}\). We refer to \cite{E05} for details.

The unifying feature of the examples appearing in \cref{THM: bundles over other spaces} is that the base manifolds come equipped with a principal \(S^1\)-bundle that carries an invariant metric of non-negative sectional curvature, and whose associated complex line bundle generates the Picard group of the base manifold.  The idea is then to show that any real vector bundle is stably equivalent to a sum of at most \(r/2\) complex line bundles.  See \cref{THM: nnc on vector bundles over type s} below for a general form of \cref{THM: bundles over Esch,THM: bundles over other spaces}.

Note that there are infinitely many manifolds in \cref{THM: bundles over Esch,THM: bundles over other spaces} that are not diffeomorphic to homogeneous spaces.
Indeed,  there are infinitely many spaces among Eschenburg and generalized Witten spaces that are not even homotopy equivalent to any homogeneous space \cite{Sha02, E05}.

\subsection*{Outline} The paper is organized as follows. All theorems above follow from a study of stable equivalence classes of real vector bundles over manifolds of dimension at most seven, with which we begin in \cref{sec:KO}.  \Cref{THM: bundles over Esch,THM: bundles over other spaces} are deduced in \cref{sec:curvature}.  In \cref{sec: Esch}, we use the results on stable equivalence classes to refine the homotopy classification of Eschenburg spaces due to Kruggel, Kreck and Stolz to a classification up to tangential homotopy equivalence.  A search for pairs as in \cref{THM: existence of tangential nonhomeo pairs} can then easily be implemented as a computer program. The code we use is briefly discussed at the end of \cref{sec: Esch}; we have made it freely available~\cite{zenodo}. \Cref{THM: nondiffeo souls with positive curv} is finally proved in \cref{sec: Proof of THM A}.  We close in \cref{sec: mmoduli spaces} with a brief discussion of implications for moduli spaces.

\subsection*{Notation} We write \(H^*(-)\) to denote (singular) cohomology with integral coefficients, i.e.\ \(H^*(X) := H^*(X,\ZZ)\).


\subsection*{Acknowledgements} This work grew out of a question by Wilderich Tuschmann to the first author, during the conference ``Curvature and Global Shape'' celebrated in M\"unster in July 2017.   We would like to thank Christine Escher for sharing unabridged versions of the lists of pairs of Eschenburg spaces published in \cite{CEZ07} with us, and Igor Belegradek and Anand Dessai for helpful comments on a preliminary version of this manuscript.  Numerous improvements where moreover suggested by an extremely diligent anonymous referee.


\section{Vector bundles over seven-manifolds}\label{sec:KO}

Two real vector bundles \(F\) and \(F'\) over a common base \(X\) are \textbf{stably equivalent} if \(F\oplus \RR^k\cong F' \oplus \RR^{k'}\) for certain integers \(k\) and \(k'\).
The main result of this section is that, over certain classes of \(7\)-manifolds, any real vector bundle is stably equivalent to a sum of complex line bundles.  See \cref{prop: sum of line bundles} for the precise statement and \cref{Rem: generalization higher dim} for slight generalizations.

Our calculations will make use of the \textbf{Spin characteristic class} \(q_1\) constructed by Thomas \cite{Thomas}.  Assume for the following brief discussion that our base \(X\) is a finite-dimensional connected CW complex.  A Spin bundle \(F\) over \(X\) is a real vector bundle whose first two Stiefel-Whitney classes \(w_1 F\) and \(w_2F\) vanish.  Equivalently, a real vector bundle \(F\) is a Spin bundle if and only if its classifying map \(f_F\colon X\to BO\) lifts to a map \(\hat f_F\colon X\to BSpin\).  The Spin characteristic class \(q_1 F \in H^4(X)\) of such a Spin bundle is defined as the pullback under \(\hat f_F\) of a distinguished generator of \(H^4(BSpin)\).  We will make frequent use of the following properties of the Spin characteristic class and its relation to the first Pontryagin class \(p_1\) and the Chern classes~\(c_1\) and \(c_2\).

\begin{proposition}\label{prop:q1}
 Let \(F\) and \(F'\) be two Spin bundles over \(X\), let \(E\) be a complex vector bundle over \(X\), and let \(rE\) be the underlying real vector bundle.
  \begin{compactenum}[(a)]
  \item \label{eq:q1:0}
    \(\phantom{2}q_1(F) = 0\)
    \quad if \(F\) is a trivial vector bundle.
  \item \label{eq:q1:+}
    \(\phantom{2}q_1(F\oplus F') = q_1 F +q_1 F'\)
  \item \label{eq:q1:p1}
    \(2q_1(F) = p_1(F)\)
    \quad --- ``The Spin class is half the Pontryagin class.''
  \item \label{eq:q1:2q1r}
    \(\phantom{2}p_1(rE) = (c_1E)^2-2c_2 E\)
  \item \label{eq:q1:q1r}
    \(\phantom{2}q_1(rE) = -c_2 E\)
    \quad if \(c_1 E = 0\).
  \end{compactenum}
\end{proposition}
For the last identity, note the \(rE\) is a Spin bundle if and only if the mod-2-reduction of \(c_1 E\) in \(H^2(X,\ZZ_2)\) vanishes.
In particular, the stated stronger condition \(c_1 E = 0\) implies that \(rE\) is a Spin bundle.
\begin{proof}
  The first claim is clear from the definition.
  For \eqref{eq:q1:+} and \eqref{eq:q1:p1}, see eqs~1.10 and 1.5 in Thm~1.2 of \cite{Thomas}.
  Claim~\eqref{eq:q1:2q1r} is a direct consequence of the definition of Pontryagin classes.
  Claim~\eqref{eq:q1:q1r} is immediate from \eqref{eq:q1:p1} and \eqref{eq:q1:2q1r} when \(H^4(X)\) contains no \(2\)-torsion, an assumption we will frequently make below.  To see that \eqref{eq:q1:q1r} also holds in general, note that stable equivalence classes of bundles with vanishing first Chern class are classified by \(BSU\).  So \(q_1\circ r\) defines a natural transformation \([X,BSU]\to H^4(X)\) and hence corresponds to an element of \(H^4(BSU)= \ZZ c_2\).
  To see which element it is, we can evaluate, say, on \(X=\mathbb{S}^4\) and then use~\eqref{eq:q1:2q1r}.
\end{proof}

\begin{proposition}\label{prop:w2p1}
  Suppose \(X\) is a connected CW complex of dimension~\(\leq 7\).  Then two Spin bundles \(F\), \(F'\) over \(X\) are stably equivalent if and only if their Spin characteristic classes agree.

  Suppose in addition that \(H^4(X)\) contains no \(2\)-torsion.  Then two real bundles \(F\), \(F'\) over \(X\) are stably equivalent if and only if their Stiefel-Whitney classes \(w_1\) and \(w_2\) and their first Pontryagin classes \(p_1\) agree.
\end{proposition}
The distinction of cases here is necessary because, in contrast to \(w_1\), \(w_2\) and \(p_1\), the Spin characteristic class \(q_1\) is not defined for arbitrary real vector bundles.
\begin{proof}
  Let \(\rKO(X)\) denote the reduced real K-group of \(X\), \ie the group of stable equivalence classes of real vector bundles over \(X\). (For background, see for example \cite{Hu}.)  Let \(\rKSpin(X)\) denote the subgroup of stable equivalence classes of Spin bundles.
  Equations \eqref{eq:q1:0} and \eqref{eq:q1:+} of the previous proposition show that \(q_1\) defines a homomorphism \(q_1\colon \rKSpin(X)\to H^4(X)\).
  By \cite[Cor.~1]{LiDuan}, this homomorphism is an isomorphism for \(X\) of dimension at most seven, so the claim follows.

  In general, \(\rKSpin(X)\) and \(\rKO(X)\) fit into a short exact sequence as follows \cite[Rem.~2]{LiDuan}:
  \[
    0 \to \rKSpin(X) \lhook\joinrel\longrightarrow \rKO(X) \xrightarrow{(w_1,w_2)} H^1(X,\ZZ_2) \times H^2(X,\ZZ_2) \to 0
  \]
  Here, the group structure on \(H^1(X,\ZZ_2) \times H^2(X,\ZZ_2)\) is defined such that the map \((w_1,w_2)\) is a homomorphism. Given two real vector bundles \(F\) and \(F'\) whose Stiefel-Whitney classes \(w_1\) and \(w_2\) agree, we obtain an element \(F-F'\in \rKO(X)\) that lies in the kernel of \((w_1,w_2)\) and hence in \(\rKSpin(X)\).  If furthermore \(p_1(F) = p_1(F')\), we find that \(p_1(F-F') = 0\) because the Whitney sum formula holds for Pontryagin classes up to \(2\)-torsion \cite[Thm~15.3]{MilnorStasheff} and because we have assumed that \(H^4(X)\) does not contain any such torsion.
  Using \cref{prop:q1}\,\eqref{eq:q1:p1} and the same assumption on \(H^4(X)\),  we deduce that \(q_1(F-F')=0\).  As we saw in the first part of the proof, this implies that \(F-F'=0\) in \(\rKSpin(X)\).   So \(F\) and \(F'\) are stably equivalent.
\end{proof}
As \(q_1\) is a homeomorphism invariant \cite[1.1\slash Rem.~2.1]{CN}, and as Stiefel-Whitney classes are even homotopy invariants, the above proposition implies:
\begin{corollary}\label{rem: homeo type s mfds are tangentially homotopic}
 Any two homeomorphic closed Spin manifolds of dimension~\(\leq 7\) are tangentially homotopy equivalent.  Similarly, any two homeomorphic closed manifolds of dimension~\(\leq 7\) for which \(H^4(-)\) contains no \(2\)-torsion are tangentially homotopy equivalent.
\end{corollary}

\medskip

We introduce the following notation for a CW complex \(X\) with \(H^4(X)\) finite:
\begin{equation}\label{eq:def:sigma4}
  \sigma_4(X) = \begin{cases*}
    1 & if $H^4(X)=0$ \\
    4 & if $\abs{H^4(X)}$ is odd \\
    9 & if $\abs{H^4(X)}$ is even and non-zero
  \end{cases*}
\end{equation}

\begin{proposition}\label{prop: sum of line bundles}
  Let \(X\) be a connected CW complex of dimension \(\leq 7\) such that \(H^1(X,\ZZ_2)=0\), \(H^2(X)\) is (non-zero) cyclic, \(H^3(X)\) contains no \(2\)-torsion, and \(H^4(X)\) is finite cyclic and generated by the square of a generator of \(H^2(X)\).%
  Then any real vector bundle over \(X\) is stably equivalent to (the underlying real bundle of) a Whitney sum of \(\sigma_4(X)\) complex line bundles.
\end{proposition}

\begin{proof}
  Under our assumptions, the Bockstein sequence shows that the reduction map \(H^2(X)\to H^2(X,\ZZ_2)\) is surjective, and that either \(H^2(X,\ZZ_2)=0\) or \(H^2(X,\ZZ_2)\cong \ZZ_2\).  We identify \(H^4(X)\) with \(\ZZ_s\) for some positive integer~\(s\).  We will not distinguish between integers and their images in any of these residue groups notationally.  Given an integer \(a\), we write \(L_a\) for the complex line bundle with \(c_1(L_a) = a \in H^2(X)\).  More generally, a sum of such line bundles will be denoted \(L_{a_1,\dots, a_k} := L_{a_1} \oplus \cdots \oplus L_{a_k}\).

  If \(s=0\), \ie if \(H^4(X)\) vanishes, then by the second half of \cref{prop:w2p1} the stable equivalence class of a real vector bundle \(F\) over \(X\) is determined by \(w_2(F)\).  Thus \(F\) is stably equivalent to either \(r(L_0)\) or \(r(L_1)\).

  Next, consider the case that the order \(s\) of \(H^4(X)\) is odd.  Let \(F\) be an arbitrary given real vector bundle over \(X\). By the second part of \cref{prop:w2p1}, it suffices to find integers \(a_1, \dots, a_4\) such that
  \begin{align*}
     w_2(rL_{a_1,\dots,a_4}) &= w_2(F) \tag{i} \label{eq:sum-of-lbs:w2}\\
     p_1(rL_{a_1,\dots,a_4}) &= p_1(F)  \tag{ii} \label{eq:sum-of-lbs:p1}
  \end{align*}
  If \(H^2(X,\ZZ_2)=0\), we can ignore the first condition; otherwise, \(w_2(rL_{a_1,\dots,a_4}) = a_1+a_2+a_3+a_4 \mod 2\).  For the Pontryagin class, part~\eqref{eq:q1:2q1r} of \cref{prop:q1} implies that
    \begin{align*}
    p_1(rL_{a_1,\dots,a_4}) &= a_1^2 + a_2^2 + a_3^2 + a_4^2 \quad \in H^4(X).
  \end{align*}
  So we can find integers \(a_i\) satisfying condition~\eqref{eq:sum-of-lbs:p1} by appealing to Lagrange's Four Square Theorem:  any positive integer can be written as a sum of a most four squares.  In case these integers do not already satisfy condition \eqref{eq:sum-of-lbs:w2}, we can replace \(a_1\) by \(a_1+s\):  as \(a_1+s = a_1+1 \mod 2\) and \((a_1+s)^2 = a_1^2 \mod s\), the new set of integers will then satisfy both conditions.

  Finally, for arbitrary \(s\), we can argue as follows.  Let \(F\) again be some given real vector bundle over \(X\), but assume to begin with that \(F\) is a Spin bundle.  Then in view of \cref{prop:w2p1} it suffices to show that there exists a Whitney sum of (at most nine) complex line bundles \(L_{a_1,\dots, a_k}\) such that \(rL_{a_1,\dots, a_k}\) is a Spin bundle with the same Spin characteristic class as~\(F\).  As the first Chern class of such a sum is given by
  \[
    c_1(L_{a_1,\dots,a_k}) = a_1 + \cdots + a_k,
  \]
  \(rL_{a_1,\dots,a_k}\) is certainly a Spin bundle whenever \(a_1 + \cdots + a_k \equiv 0 \mod 2\).  Moreover, part~\eqref{eq:q1:q1r} of \cref{prop:q1} applies whenever \(a_1 + \cdots + a_k =0\) in \(\ZZ\). In particular, we find that \(q_1(rL_{a,-a})= a^2\), and more generally that
  \[
    q_1(rL_{a_1,-a_1,a_2,-a_2,a_3,-a_3,a_4,-a_4}) = a_1^2 + a_2^2 + a_3^2 + a_4^2 \quad \in H^4(X).
  \]
  So, again by Lagrange's Four Square Theorem, we can find integers \(a_1\), \(a_2\), \(a_3\), \(a_4\) such that  \(q_1(rL_{a_1,-a_1,\dots,a_4,-a_4}) = q_1 F\), whatever the given value of \(q_1 F\).  So our Spin bundle \(F\) is stably equivalent to a Whitney sum of eight complex line bundles.

  When \(F\) is an arbitrary real vector bundle, we can pick a complex line bundle \(L_b\) such that \(w_2(rL_b) = w_2(F)\).  Then \(F-rL_b\) is a stable equivalence class in \(\rKSpin(X)\), the previous argument shows that \(F-rL_b = rL_{a_1,-a_1,\dots,a_4,-a_4}\) in \(\rKSpin\), and hence \(F\) is stably equivalent to the Whitney sum of nine complex line bundles \(rL_{a_1,-a_1,\dots,a_4,-a_4,b}\).
\end{proof}

\begin{corollary}\label{cor: sum of line bundles}
  Let \(X\) be a connected CW complex satisfying the assumptions of \cref{prop: sum of line bundles}.
  Any real vector bundle over \(X\) of rank \(\geq \max\{2\sigma_4(X),\dim(X) + 1\}\) is isomorphic to a Whitney sum of \(\sigma_4(X)\) complex line bundles and a trivial bundle.
\end{corollary}

\begin{proof}
This is immediate from \cref{prop: sum of line bundles} and the general fact that the notions of stable equivalence and isomorphism agree for bundles of sufficiently high rank:
  if two real vector bundles of the same rank \(F\) and \(F'\) over an \(n\)-dimensional CW complex are stably equivalent, and if the common rank of these bundles is greater than \(n\),  then \(F\) and \(F'\) are isomorphic (e.g. \cite[Ch.~9, Prop.~1.1]{Hu}).
\end{proof}

\begin{remark}\label{Rem: generalization higher dim}
  We have deliberately refrained from stating \cref{prop:w2p1,prop: sum of line bundles,cor: sum of line bundles} with minimal assumptions.
  In \cref{prop:w2p1}, the condition that \(X\) is a connected CW complex of dimension \(\leq 7\) could easily be replaced with the following weaker assumptions:

  \begin{itemize}[--]
  \item \(X\) is a connected finite-dimensional CW complex.
  \item The inclusion of the seven-skeleton \(X^7\) induces an isomorphism \(\rKO(X^7)\cong \rKO(X)\).  The Atiyah-Hirzebruch spectral sequence shows that a sufficient criterion for this to be the case is that all non-vanishing integral cohomology groups \(H^i(X)\) in degrees \(i\geq 5\) are torsion-free and concentrated in degrees \(i\) with \((i\mod 8)\in\{3,5,6,7\}\).
\end{itemize}
  The additional assumptions needed in \cref{prop: sum of line bundles,cor: sum of line bundles} are that \(H^1(X,\ZZ_2)\), \(H^2(X)\), \(H^3(X)\) and \(H^4(X)\) have the properties stated in \cref{prop: sum of line bundles}.
\end{remark}

\section{Non-negative curvature}\label{sec:curvature}
In this section we review a common construction of non-negatively curved metrics on vector bundles and prove \cref{THM: bundles over Esch,THM: bundles over other spaces}, which give partial positive answers to the converse soul question for Eschenburg spaces and a few other spaces.

\medskip

Let $G$ be a Lie group and let $P\to M$ be a principal $G$-bundle. Given a representation $\rho\colon G\to \RR^m$, there exists a natural diagonal action on the product $P\times\RR^m$ whose quotient space $E_\rho = P\times_G \RR^m$ is the total space of a real vector bundle over $M$.
This constructions yields a natural semiring homomorphism:
\[
\Rep(G)\to \Vect(M)
\]
Suppose now that $P$ admits a $G$-invariant metric $g_P$ with non-negative sectional curvature. By the Gray-O'Neill formula for Riemannian submersions, $M$ inherits a metric \(\bar g_P\) with non-negative sectional curvature. Now suppose that \(\rho\colon G\to \RR^m\) is an orthogonal representation with respect to the usual Euclidian metric $g_0$ on $\RR^m$.  Equip \(P\times \RR^m\) with the product metric $g_P\times g_0$.  Then $P\times\RR^m$ also has non-negative sectional curvature and the diagonal $G$-action on $P\times\RR^m$ is by isometries. So, again by the Gray-O'Neill formula, $E_\rho$ inherits a metric with non-negative sectional curvature for which the zero-section $(P\times_G\{0\},\bar g_p)=(M,\bar g_P)$ is a soul.

At the present time, this is the only known construction of open manifolds with non-negative sectional curvature, up to a change of metric (see \cite[Section 3.1]{Wi}). It is natural to ask which vector bundles over $M$ can be constructed in this way, a purely topological question that is discussed at length in \cite{GoZ} for the case when \(P\to M\) is the canonical \(G\)-bundle over a homogeneous space \(G'/G\).  Here, we consider circle bundles, i.\,e.\ the case $G=S^1$.

\begin{proposition}\label{PROP:metrics type r}
Let \(P\to M\) be a principal circle bundle over a closed manifold~\(M\).
Assume that \(P\) is \(2\)-connected and that it admits an invariant metric \(g_P\) of non-negative sectional curvature. Then the total space of any Whitney sum of complex line bundles over \(M\) admits a metric of non-negative sectional curvature and with soul isometric to \((M,\bar g_P)\), where \(\bar g_P\) denotes the quotient metric inherited from \(g_P\).
\end{proposition}
\begin{proof}
As explained in \cite[Section 12]{BKS15}, the fact that \(P\) is \(2\)-connected implies that \(H^2(M)=\ZZ\) and that the first Chern class of the bundle is a generator of \(H^2(M)\). It follows that any complex line bundle over \(M\) has the form \(E_\rho = P\times_{S^1}\CC\) for some character \(\rho\) of \(S^1\), and more generally that any Whitney sum of complex line bundles has the form  \(E_\rho = P\times_{S^1}\CC^k\) for some direct sum of characters \(\rho\in \Rep(S^1)\).  So the claim follows immediately from the discussion above.
\end{proof}
Conditions for a circle bundle to admit invariant metrics with non-negative sectional curvature are given in \cite{STT}.

\Cref{THM: bundles over Esch,THM: bundles over other spaces} of the introduction are particular cases of the following more general statement.
Recall from eq.\,\eqref{eq:def:sigma4} in \cref{sec:KO} our notation \(\sigma_4(M)\) for a space with \(H^4(M)\) finite.

\begin{proposition}\label{THM: nnc on vector bundles over type s}
Let \(P\to M\) be a principal circle bundle over a closed manifold~\(M\). Assume that
\begin{compactitem}
\item \(P\) is \(2\)-connected (so that \(H^1(M)=0\) and \(H^2(M)=\ZZ\)) and that it admits an invariant metric \(g_P\) of non-negative sectional curvature, and that
\item \(H^3(M)\) contains no \(2\)-torsion, \(H^4(M)\) is finite cyclic and generated by the square of a generator of \(H^2(M)\), and  all non-vanishing integral cohomology groups \(H^i(M)\) in degrees \(i\geq 5\) are torsion-free and concentrated in degrees \(i\) with \((i\mod 8)\in\{3,5,6,7\}\).
\end{compactitem}
Then the total space of every real vector bundle of rank \(\geq {\max\{2\sigma_4(M),\dim(M)+1\}}\)  over \(M\)  admits a metric with non-negative sectional curvature and soul isometric to \((M,\bar{g}_P)\), where \(\bar g_P\) is the induced quotient metric on \(M\).
\end{proposition}
\begin{proof}
  \Cref{cor: sum of line bundles,Rem: generalization higher dim} show that any real vector bundle \(F\) over \(M\) of rank \(\geq \max\{2\sigma_4(M),\dim(M)+1\}\) is isomorphic to a Whitney sum of complex line bundles and a trivial vector bundle.  The Whitney sum of complex line bundles admits a metric of non-negative sectional curvature by \cref{PROP:metrics type r}, and thus the product metric of this metric with the flat metric on the trivial summand yields a metric on \(F\) with the desired properties.
\end{proof}

To prove \cref{THM: bundles over Esch,THM: bundles over other spaces}, it now suffices to check that the spaces in question satisfy the assumptions of \cref{THM: nnc on vector bundles over type s}.
\begin{proof}[Proof of \cref{THM: bundles over Esch,THM: bundles over other spaces}]
The cohomology of Eschenburg and generalized Witten spaces is well known \cite{E05,Esch82}: they are manifolds of type \(r\) (see \cref{DEF: type r} below). For Eschenburg spaces \(|H^4(M)|\) is odd, while for generalized Witten spaces it can be either odd or even so both \(\sigma_4(M)=4\) and \(\sigma_4(M)=9\) occur. The total spaces of the corresponding principal bundles are \(SU(3)\) and \(\mathbb{S}^3\times\mathbb{S}^5\), respectively, which clearly satisfy the topological assumptions of \cref{THM: nnc on vector bundles over type s}. The corresponding metrics on \(SU(3)\) were constructed by Eschenburg \cite{Esch82}, see \cref{sec: Esch} below. As for the generalized Witten spaces, the circle actions are by isometries with respect to the standard product metric on \(\mathbb{S}^3\times\mathbb{S}^5\) (see \cite{E05}).

The products \(\mathbb{S}^2\times\mathbb{S}^m\) and the unique non-trivial \(\mathbb{S}^m\)-bundle over \(\mathbb{S}^2\) with \(m\geq 2\)
have the same cohomology ring,
which clearly satisfies the topological assumptions when \(m\equiv 3,5\mod 8\). The products \(\mathbb{S}^2\times \mathbb{S}^m\) are just quotients of \(\mathbb{S}^3\times \mathbb{S}^{m}\) via the Hopf fibration over the first factor. The unique non-trivial linear \(\mathbb{S}^m\)-bundle over \(\mathbb{S}^2\) with \(m=3\) or \(m\equiv 5\mod 8\) can be described as a circle quotient of \(\mathbb{S}^3\times \mathbb{S}^{m}\) as well. Moreover, the corresponding action is by isometries with respect to the standard product metric on \(\mathbb{S}^3\times \mathbb{S}^{m}\):  see \cite{DeV} for the case \(m=3\) and \cite[item (b) above Corollary~4]{WZ} for the cases \(m\equiv 5\mod 8\).%
\end{proof}


\section{Eschenburg spaces}\label{sec: Esch}
\renewcommand{\r}{r} 
Eschenburg spaces, first introduced and studied in \cite{Esch82}, generalize the homogeneous \(7\)-manifolds known as Aloff-Wallach spaces.  Each Eschenburg space is a quotient of \(SU(3)\) by a free action of \(S^1\) of the following form:
\begin{align*}
  S^1 \times SU(3) & \longrightarrow SU(3)\\
(z,A) &\mapsto \diag(z^{k_1},z^{k_2},z^{k_3}) \cdot A \cdot \diag(z^{-l_1},z^{-l_2},z^{-l_3})
\end{align*}
Following \cite{CEZ07}, we specify the action of \(S^1\) and the resulting Eschenburg space \(M = M(k,l)\) by the six-tuple of integer parameters \((k,l) = (k_1,k_2,k_3,l_1,l_2,l_3)\).  We refer to this six-tuple as the parameter vector of \(M\). The parameters need to satisfy \(k_1+k_2+k_3 = l_1+l_2+l_3\), as well as some further conditions that ensure that the \(S^1\)-action is free, see \cite[(1.1)]{CEZ07}.  The Aloff-Wallach spaces are the Eschenburg spaces \(M(k,l)\) with \(l_1 = l_2 = l_3 =0\).%

All Aloff-Wallach spaces \(M(k,0)\) with \(k_1 k_2 k_3\neq 0\) admit an invariant metric of positive sectional curvature.
The interest in more general Eschenburg spaces arises from the fact that they include some of the very few known examples of \emph{non}-homogeneous manifolds with positive sectional curvature.
Any metric on \(\SU(3)\) invariant under the circle action defined by \((k,l)\) descends to a metric on the Eschenburg space \(M(k,l)\). We refer to a metric on an Eschenburg space arising in this way as a \emph{submersion metric}.  Every Eschenburg space comes equipped with non-negatively curved submersion metrics.  For example, one could consider metrics induced by  bi-invariant metrics on \(\SU(3)\), but there are also lots of other choices.
Eschenburg constructed submersion metrics with \emph{positive} sectional curvature on infinitely many Eschenburg spaces \cite[Satz~414]{Esch84}. In particular, he did so for all Eschenburg spaces \(M(k,l)\) whose parameter vector satisfies the following condition:
\begin{equation}
\tag{\(\dagger\)}\label{eq:Esch:standard-parameters}
k_1 \geq k_2 > l_1 \geq l_2 \geq l_3 = 0.
\end{equation}
In fact, as explained in \cite[Lemma~1.4]{CEZ07}, each of the Eschenburg spaces on which Eschenburg constructed a positively curved submersion metric is diffeomorphic to one of the spaces \(M(k,l)\) satisfying \eqref{eq:Esch:standard-parameters}.

Positively curved Eschenburg spaces display interesting phenomena that are not visible when studying the Aloff-Wallach subfamily alone.  The following proposition is one example of this.  Part~(b) was already stated as \cref{THM: existence of tangential nonhomeo pairs} of the introduction.

\begin{proposition}\label{prop:Eschenburg-phenomena}
For Aloff-Wallach spaces, the notions of homotopy equivalence, tangential homotopy equivalence and homeomorphism coincide.
In contrast, for general positively curved Eschenburg spaces, these notions differ:
 \begin{itemize}
 \item[(a)] There exist pairs of positively curved Eschenburg spaces which are homotopy equivalent to each other but not tangentially homotopy equivalent.
 \item[(b)] There exist pairs of positively curved Eschenburg spaces which are tangentially homotopy equivalent but not homeomorphic.
 \end{itemize}
Examples of both phenomena are displayed in \cref{table:Esch:phenomena}.
\end{proposition}

For Aloff-Wallach spaces, the equivalence of the notions of homotopy equivalence and homeomorphism is due to Dickinson and Shankar \cite{Sha02}.  Slightly weaker versions of the statements for positively curved Eschenburg spaces, namely the existence of pairs of positively curved Eschenburg spaces which are homotopy equivalent but not homeomorphic,  is known by \cite{Sha02,CEZ07}.  Also, there are known pairs of positively curved Eschenburg spaces \cite[Table~2]{CEZ07} and even of Aloff-Wallach spaces \cite[Corollary on p.\,467]{KS91} which are homeomorphic but not diffeomorphic.  The situation is illustrated in \cref{fig:iso-implications}.
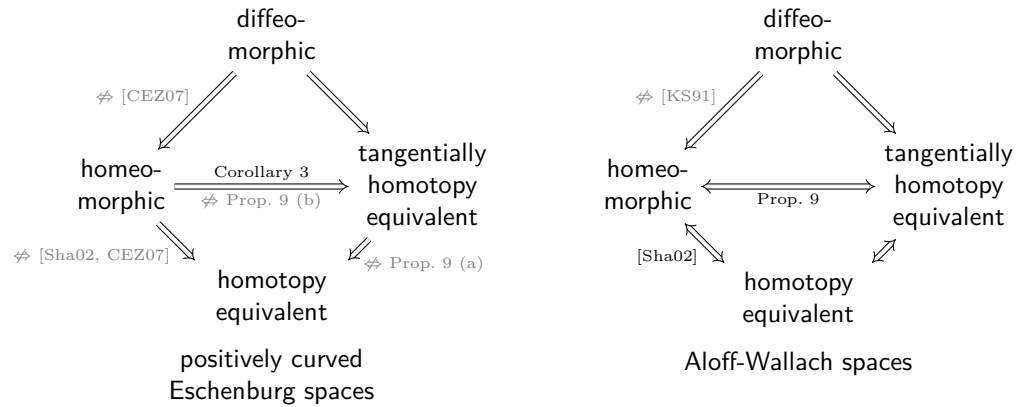
\begin{figure}[b]
  \newcommand*{\strict}[1]{\textcolor{gray}{$\not\hspace{-0.5pt}\Leftrightarrow$ #1}}

\begin{tikzpicture}[x=2cm,y=2cm,commutative diagrams/every diagram]
  \path[align=center,font=\sffamily]
  node at (0,1) (diff) {diffeo-\\morphic}
  node at (-1,0) (homeo) {homeo-\\morphic}
  node at (1,0) (the) {tangentially\\homotopy\\equivalent}
  node at (0,-1) (he) {homotopy\\equivalent\\[6pt]positively curved\\Eschenburg spaces};

  \path[commutative diagrams/.cd,every arrow,every label,font=\tiny]
  (diff)  edge[commutative diagrams/Rightarrow] node[swap] {\strict{\cite{CEZ07}}} (homeo)
  (diff)  edge[commutative diagrams/Rightarrow] (the)
  (homeo) edge[white,text=black] node {\Cref{rem: homeo type s mfds are tangentially homotopic}} (the)
  (homeo) edge[commutative diagrams/Rightarrow] node[swap] {\strict{Prop.~\ref{prop:Eschenburg-phenomena} (b)}} (the)
  (the)   edge[commutative diagrams/Rightarrow] node {\strict{Prop.~\ref{prop:Eschenburg-phenomena} (a)}}
  (he)
  (homeo) edge[commutative diagrams/Rightarrow] node[swap] {\strict{\cite{Sha02,CEZ07}}} (he);

  \path[align=center,font=\sffamily]
  node at (3.5,1) (diff) {diffeo-\\morphic}
  node at (2.5,0) (homeo) {homeo-\\morphic}
  node at (4.5,0) (the) {tangentially\\homotopy\\equivalent}
  node at (3.5,-1) (he) {homotopy\\equivalent\\[6pt]Aloff-Wallach spaces\\~};

  \path[commutative diagrams/.cd,every arrow,every label,font=\tiny]
  (diff)  edge[commutative diagrams/Rightarrow] node[swap] {\strict{\cite{KS91}}} (homeo)
  (diff)  edge[commutative diagrams/Rightarrow] (the)
  (homeo) edge[white,text=black] 
  (the)
  (homeo) edge[commutative diagrams/Leftrightarrow] node[swap] {Prop.~\ref{prop:Eschenburg-phenomena}}
  (the)
  (the)   edge[commutative diagrams/Leftrightarrow] 
  (he)
  (homeo) edge[commutative diagrams/Leftrightarrow] node[swap] {\cite{Sha02}}
  (he);
\end{tikzpicture}
  \caption{Implications between different notions of isomorphism for positively curved Eschenburg spaces and for the subfamily of Aloff-Wallach spaces, respectively.  All indicated implications \((\Rightarrow)\) are strict. The references in gray refer to counterexamples to the inverse implications.}\label{fig:iso-implications}
\end{figure}

\medskip

\newcolumntype{C}{>{$}c<{$}}
\newcolumntype{R}{>{$}r<{$}}
\newcolumntype{L}{>{$}l<{$}}
\newcolumntype{P}[1]{>{\raggedleft\arraybackslash}p{3ex}}
\newlength{\pairskip}
\setlength{\pairskip}{6pt}
\newlength{\tableaddwidth}
\setlength{\tableaddwidth}{4em}
\newcommand{\minus}{-}
\newcommand{\plus}{\phantom{+}}

\begin{table}
  \begin{adjustwidth}{-\tableaddwidth}{-\tableaddwidth}\centering
  \begin{tabular}{r@{$(\,$}R@{, }R@{, }R@{, }R@{, }R@{, \;\;}R@{$\,)$ \quad}RRRRRCC}
    \toprule
      & k_1  & k_2  & k_3  & l_1  & l_2  & l_3 & \multicolumn{1}{C}{\r} & \multicolumn{1}{C}{\quad s} & \multicolumn{1}{C}{\Sigma} & p_1     & s_{22}                 & s_{2}                             \\
    \midrule
    \multicolumn{12}{l}{Homotopy equivalent but not tangentially homotopy equivalent:}                                                                                                                    \\[\pairskip]
      & 8    & 7    & -5   & 6    & 4    & 0   & 43                     & -21                         & 1                         & 13      & \nicefrac{\plus  1}{6} & \nicefrac{\minus  59}{516} \\
      & 21   & 21   & -2   & 20   & 20   & 0   & 43                     & -21                         & 1                         & 26      & \nicefrac{\plus  1}{6} & \nicefrac{\plus   55}{516}        \\[\pairskip]
      & 12   & 10   & -8   & 9    & 5    & 0   & 101                    & -50                         & -1                        & 21      & \nicefrac{\plus  1}{6} & \nicefrac{\plus  565}{1212}       \\
      & 50   & 50   & -2   & 49   & 49   & 0   & 101                    & -50                         & -1                        & 55      & \nicefrac{\plus  1}{6} & \nicefrac{\minus 125}{1212}       \\[\pairskip]
      & 19   & 17   & -7   & 16   & 13   & 0   & 137                    & -68                         & -1                        & 23      & \nicefrac{\plus  1}{6} & \nicefrac{\minus 743}{1644}       \\
      & 68   & 68   & -2   & 67   & 67   & 0   & 137                    & -68                         & -1                        & 73      & \nicefrac{\plus  1}{6} & \nicefrac{\plus  241}{1644}       \\[\pairskip]
      & 30   & 26   & -6   & 25   & 25   & 0   & 181                    & -26                         & -1                        & 164     & \nicefrac{\minus 1}{6} & \nicefrac{\minus 193}{2172}       \\
      & 16   & 16   & -10  & 13   & 9    & 0   & 181                    & 26                          & 1                         & 85      & \nicefrac{\plus  1}{6} & \nicefrac{\minus 443}{2172}       \\[\pairskip]
      & 15   & 14   & -11  & 12   & 6    & 0   & 181                    & -43                         & 0                         & 35      &                      0 & \nicefrac{\minus  55}{181} \\
      & 45   & 43   & -4   & 42   & 42   & 0   & 181                    & -43                         & 0                         & 89      &                      0 & \nicefrac{\plus   36}{181}        \\[\pairskip]
      & 16   & 13   & -11  & 12   & 6    & 0   & 183                    & -91                         & 0                         & 33      & \nicefrac{\minus 1}{6} & \nicefrac{\minus 991}{2196}       \\
      & 91   & 91   & -2   & 90   & 90   & 0   & 183                    & -91                         & 0                         & 96      & \nicefrac{\minus 1}{6} & \nicefrac{\plus  413}{2196}       \\
    \multicolumn{6}{C}{\vdots}                                                                                                                                                                            \\[\pairskip]
    \midrule
    \multicolumn{12}{l}{Tangentially homotopy equivalent but not homeomorphic:}                                                                                                                           \\[\pairskip]
      & 58   & 54   & -34  & 39   & 39   & 0   & 2\,197                 & 1\,032                      & 0                          & 845     & \nicefrac{\plus 1}{2} & \nicefrac{\plus   1147}{8788}     \\
      & 45   & 41   & -47  & 39   & 0  & 0   & 2\,197                 & 1\,032                      & 0                          & 845     & \nicefrac{\plus 1}{2} & \nicefrac{\minus  3247}{8788}     \\[\pairskip]
      & 81   & 69   & -84  & 56   & 10  & 0   & 7\,571                 & 74                          & 0                          & 5\,352  & \nicefrac{\plus 1}{2} & \nicefrac{\minus  9219}{30284}    \\
      & 108  & 63   & -69  & 56   & 46  & 0   & 7\,571                 & 74                          & 0                          & 5\,352  & \nicefrac{\plus 1}{2} & \nicefrac{\plus   5923}{30284}    \\[\pairskip]
      & 88   & 61   & -107 & 30   & 12  & 0   & 10\,935                & -5\,179                     & 0                          & 1\,368  & \nicefrac{\minus 1}{6} & \nicefrac{\plus  55529}{131220}   \\
      & 77   & 77   & -106 & 30   & 18  & 0   & 10\,935                & 5\,179                      & 0                          & 1\,368  & \nicefrac{\plus 1}{6} & \nicefrac{\minus 11789}{131220}   \\[\pairskip]
      & 79   & 58   & -131 & 6   & 0  & 0   & 13\,365                & -1\,183                     & 0                          & 72      & \nicefrac{\plus 1}{3} & \nicefrac{\minus  3794}{8019}     \\
      & 92   & 47   & -127 & 6   & 6  & 0   & 13\,365                & 1\,183                      & 0                          & 72      & \nicefrac{\minus 1}{3} & \nicefrac{\minus  1552}{8019}     \\[\pairskip]
      & 115  & 79   & -116 & 72   & 6  & 0   & 13\,851                & 1\,184                      & 0                          & 9\,576  & \nicefrac{\minus 1}{6} & \nicefrac{\minus 77167}{166212}   \\
      & 128  & 107  & -97  & 72   & 66  & 0   & 13\,851                & -1\,184                     & 0                          & 9\,576  & \nicefrac{\plus 1}{6} & \nicefrac{\minus 61343}{166212}   \\[\pairskip]
      & 1112 & 1111 & -13  & 1110 & 1100 & 0   & 14\,467                & 2\,246                      & -1                         & 11\,744 & \nicefrac{\minus 1}{6} & \nicefrac{\plus  68945}{173604}   \\
      & 127  & 103  & -106 & 88   & 36  & 0   & 14\,467                & -2\,246                     & 1                          & 11\,744 & \nicefrac{\plus 1}{6} & \nicefrac{\plus  17857}{173604}   \\[\pairskip]
      & 188  & 176  & -82  & 145  & 137  & 0   & 16\,625                & 3\,341                      & 0                          & 6\,608  & \nicefrac{\plus 1}{2} & \nicefrac{\minus 25007}{66500}    \\
      & 176  & 164  & -94  & 163  & 83  & 0   & 16\,625                & 3\,341                      & 0                          & 6\,608  & \nicefrac{\plus 1}{2} & \nicefrac{\plus ~8243}{66500}     \\
    \multicolumn{6}{C}{\vdots}                                                                                                                                                                            \\[\pairskip]
    \bottomrule
    \addlinespace
  \end{tabular}
  \end{adjustwidth}
  \caption{The ``first'' six pairs of homotopy equivalent but not tangentially homotopy equivalent pairs of positively curved Eschenburg spaces (top half of table), and the ``first'' six pairs of tangentially homotopy equivalent but non-homeomorphic pairs of such spaces. ``First'' means that these are the pairs of spaces satisfying \eqref{eq:Esch:standard-parameters} with smallest value of \(\r\).
  }\label{table:Esch:phenomena}
\end{table}

Given the concrete examples in \cref{table:Esch:phenomena}, \cref{prop:Eschenburg-phenomena} can be treated as an application of the classification of Eschenburg spaces.  We will first discuss this classification and then say a few words about how the examples were obtained.

Classifications of Eschenburg spaces are known up to various notions of equivalence.  Most relevant for us are the classifications up to homotopy and homeomorphism due to Kruggel \cite{Kru97,Kru98,Kru05}.  The simplest homotopy invariant used in these classifications is obtained via cohomology.  Namely, all Eschenburg spaces are type-\(r\)-manifolds in the following sense \cite[Proposition 36]{Esch82}:

\begin{definition}[\cite{Kru97}]\label{DEF: type r} A \textbf{type-\(r\)-manifold} is a simply connected closed \(7\)-manifold \(M\) whose cohomology has the following structure:
\begin{compactitem}[]
\item \(H^2(M) \cong \ZZ\), generated by some class \(u\)
\item \(H^4(M) \cong \ZZ_r\), generated by \(u^2\), for some finite integer \(r\geq 1\)
\item \(H^5(M) \cong \ZZ\), generated by some class \(v\)
\item \(H^7(M) \cong \ZZ\), generated by \(uv\)
\item \(H^d(M) = 0 \) in all other degrees \(d > 0\)
\end{compactitem}
\end{definition}
In particular, the order \(r\) of the fourth cohomology group is a homotopy invariant of Eschenburg spaces.  A homeomorphism invariant used in Kruggel's classification is the first Pontryagin class \(p_1\in H^4(M)\).  Note that we can \emph{canonically} identify \(H^4(M)\) with \(\ZZ_r\) as the generator \(u^2\) does not depend on any (sign) choices.  The additional invariants used by Kruggel are the linking number and certain invariants \(s_i\) developed by Kreck and Stolz for arbitrary type-\(r\)-manifolds \cite{KS93}.  Closed expressions for the Kreck-Stolz invariants of Eschenburg spaces \(M(k,l)\) are known only for spaces whose parameter vector \((k,l)\) satisfies a certain numerical ``condition~(C)'' \cite[\S\,2]{CEZ07}.  However, spaces violating this condition are relatively rare, see \cref{eg:C-failures} below.  One last homotopy invariant of positively curved Eschenburg spaces worth mentioning is the value of \(\Sigma := k_1+k_2+k_3 \mod 3\) \cite{Mil00}\cite[Prop.~12]{Sha02}.  This invariant is not used in Kruggel's classification, but it can still be useful when looking for the kind of phenomena we are studying here.

\Cref{table:invariants} attempts to give an overview over the different invariants, while \cref{table:classification} summarizes the classification results.
Note that the displayed classification of Eschenburg spaces up to \emph{tangential} homotopy equivalence is immediate from the classification up to homotopy equivalence:

\begin{table}
\begin{adjustwidth}{-\tableaddwidth}{-\tableaddwidth}\centering
\begin{tabular}{lllp{3.7cm}l}
  \toprule
                  & definition                          &                         & interpretation                                                                      & invariance          \\
  \midrule
  \(\r\)          & \(\abs{\sigma_2(k) - \sigma_2(l)}\) & \(\in\ZZ\)              & order of \(H^4(M(k,l))\)                                                            & homotopy            \\
  \(s\)           & \;\(\sigma_3(k) - \sigma_3(l)\)     & \(\in \ZZ_{\r}^\times\) & \(\nicefrac{-s^{-1}}{(\sigma_2(k) - \sigma_2(l))}\in\QQ/\ZZ\) is the linking number & or.~homotopy        \\
  \(\Sigma\)      & \;\(\sigma_1(l)\)                   & \(\in \ZZ_3\)           & --                                                                                  & or.~homotopy        \\
  \(p_1\)         & \;\(2\sigma_1(l)^2 - 6\sigma_2(l)\) & \(\in \ZZ_{\r}\)        & first Pontryagin class                                                              & tangential homotopy \\
  \(s_{22}\)      & (\(2rs_2\))                         & \(\in\QQ/\ZZ\)          & --                                                                                  & or.~homotopy        \\
  \(s_2\)         & (non-polynomial)                    & \(\in \QQ/\ZZ\)         & (Kreck-Stolz invariant)                                                             & or.~homeomorphism   \\
\bottomrule
\addlinespace
\end{tabular}
\end{adjustwidth}

\caption{Some invariants of an Eschenburg space \(M(k,l)\).  Our notation mostly follows the notation used in \cite{CEZ07}.  In the explicit formulae for the invariants, \(\sigma_i\) denotes the \(i\)th elementary symmetric polynomial, \ie \(\sigma_1(k) = k_1 + k_2 + k_3\), \(\sigma_2(k) = k_1 k_2 + k_2 k_3 + k_1 k_3\) and \(\sigma_3(k) = k_1 k_2 k_3\). The oriented invariants (``or.'') change signs under a change of orientation.
}
\label{table:invariants}
\end{table}

\begin{proposition}\label{Esch:the}
  Two Eschenburg spaces are tangentially homotopy equivalent if and only if they are homotopy equivalent and their first Pontryagin classes agree.
\end{proposition}
\begin{proof}
  The invariant \(\r\), the order of \(H^4(M)\), is odd for any Eschenburg space \(M\) \cite[above Prop.~1.7]{CEZ07}.
  In particular, \(H^4(M)\) contains no two-torsion, so that the claim follows directly from the second statement in \cref{rem: homeo type s mfds are tangentially homotopic}.
\end{proof}

\begin{proof}[Proof of \cref{prop:Eschenburg-phenomena}]
The classification results summarized in \cref{table:classification} and the examples in \cref{table:Esch:phenomena} immediately imply the claims concerning general positively curved Eschenburg spaces.

As for the statement concerning Aloff-Wallach spaces, the equivalence of the notions of homeomorphism and homotopy equivalence was proven in \cite[Proposition A.1]{Sha02}. Finally, the equivalence of the notions of homotopy equivalence and tangential homotopy equivalence follows from \cref{Esch:the} since \(p_1= 0 \) for Aloff-Wallach spaces (see \cref{table:invariants}).
\end{proof}

\begin{table}
\newcommand{\pC}{{\phantom{C}}}
\begin{adjustwidth}{-\tableaddwidth}{-\tableaddwidth}\centering
\begin{tabular}{lcll}
\toprule
 invariants {\dots} agree      & \(\Leftrightarrow\) & spaces agree up to \dots                 & references                          \\
\midrule
\(\r, s, s_{22}\)            & \(\Leftrightarrow\)   & oriented homotopy equivalence            & \cite{Kru98,CEZ07} \\
\(\r, s, s_{22}, p_1\)   & \(\Leftrightarrow\)   & oriented tangential homotopy equivalence  &  \cref{Esch:the}                    \\
 \(\r, s, s_2, p_1\)     & \(\Leftrightarrow\)   & oriented homeomorphism                   & \cite{Kru05,CEZ07} \\
\bottomrule
\addlinespace
\end{tabular}
\end{adjustwidth}
\caption{Classification of Eschenburg spaces satisfying Kruggel's condition (C), up to various notions of equivalence. For example, the first line says that two such spaces are homotopy equivalent via an orientation preserving equivalence if and only if their invariants \(\r\), \(s\), and \(s_{22}\) agree.  For a more extensive and detailed summary, see \cite[Thm~2.3]{CEZ07}.}
\label{table:classification}
\end{table}

To find the examples listed in \cref{table:Esch:phenomena}, we followed the basic strategy outlined in \cite{CEZ07}.  That is, we employed a computer program that first generates all positively curved Eschenburg spaces satisfying \eqref{eq:Esch:standard-parameters} with \(\r\) bounded by some upper bound \(R\), and then looks for families of spaces whose invariants agree.  More precisely, given an upper bound \(R\in \NN\), the main steps of the program are:
\begin{enumerate}[(1)]
\item Generate all parameter vectors \((k,l)\) satisfying \eqref{eq:Esch:standard-parameters} with \(\r \leq R\).
\item Among these parameter vectors, find all maximal families of two or more parameter vectors for which the invariants \(\r\), \(s\) and \(\Sigma\) agree, up to simultaneous sign changes of \(s\) and \(\Sigma\).
  (This intermediate step is necessary to avoid time-consuming computations of the invariant \(s_{22}\) for all generated parameter vectors.)
\item Within those families, find all maximal (sub)families of two or more parameter vectors for which, in addition, the invariant \(s_{22}\) agrees, again up to simultaneous sign changes of \(s\), \(\Sigma\) and \(s_{22}\).
  This results in a list of families of parameter vectors that describe homotopy equivalent positively curved Eschenburg spaces.
\item Within the remaining families, find all maximal (sub)families of two or more parameter vectors for which, in addition, the first Pontryagin class agrees.
  This results in a list of families of parameter vectors that describe tangentially homotopy equivalent positively curved Eschenburg spaces.
\item Within the remaining families, find all maximal (sub)families of two or more parameter vectors for which, in addition, the invariant \(s_2\) agrees (up to simultaneous sign changes of \(s\), \(\Sigma\), \(s_{22}\) and \(s_2\)).
  This results in a list of families of parameter vectors that describe homeomorphic Eschenburg spaces.
\end{enumerate}
The examples in \cref{table:Esch:phenomena} were obtained by comparing the different lists generated by the program.  Unfortunately, the C-code referred to in \cite{CEZ07} seems to have been lost, so we reimplemented the whole program from scratch and added the additional functionality we needed (in particular steps (3--5)). The new program, written completely in C\texttt{++}, is freely available \cite{zenodo}, and we encourage the reader to play around with it.  Invariants of individual spaces can alternatively be computed using some Maple code that is still available from Wolfgang Ziller's homepage.

The following empirical data obtained using the program is supplied purely for the reader's amusement.

\begin{statistics}
  Within the range of \(\r \leq 100\,000\), there are
  \begin{compactitem}[]
  \item
    101\,870\,124 -- 101\,872\,253 distinct homotopy classes,
  \item
    103\,602\,166 distinct tangential homotopy classes, and
  \item
    103\,602\,344 distinct homeomorphism classes
  \end{compactitem}
  of positively curved Eschenburg spaces satisfying~\eqref{eq:Esch:standard-parameters}.  We do not know the exact number of distinct homotopy classes due to the failure of Kruggel's condition~C in some cases.
\end{statistics}

\begin{examples}[Condition~C failures]\label{eg:C-failures}
  Examples of positively curved Eschenburg spaces for which Kruggel's condition~C fails are discussed in \cite{CEZ07}.  An example of such a space with minimal value of \(\r\) among those satisfying \eqref{eq:Esch:standard-parameters}, taken from \cite[\S2]{CEZ07}, is displayed as space \(M_0\) in \cref{table:examples}.  The spaces \((M_1,M_2)\) in \cref{table:examples} constitute a pair of positively curved Eschenburg spaces for which the invariants \(r\), \(s\), \(\Sigma\) and \(p_1\) agree, while we cannot compare the Kreck-Stolz invariants due to the failure of condition~C for one of the spaces.
  The value \(\r = 141\,151\) is minimal among all such pairs of spaces satisfying~\eqref{eq:Esch:standard-parameters}.
\end{examples}

\begin{example}[Larger exotic families]
  The literature on Eschenburg spaces only studies \emph{pairs} of exotic structures, for example pairs of homotopy equivalent spaces.  However, there also seem to be lots of triples, quadruples, etc.\ of homotopy equivalent Eschenburg spaces.  For example, the spaces \(M_3\), \(M_4\), \dots, \(M_8\) in \cref{table:examples} constitute a six-tuple of homotopy equivalent, positively curved Eschenburg spaces, no two of which are tangentially homotopy equivalent.  In contrast, we have not been able to find a single triple of tangentially homotopy equivalent but non-homeomorphic Eschenburg spaces.  There appear to be no such triples of spaces satisfying \eqref{eq:Esch:standard-parameters} with \(\r\leq 300\,000\).
  %
\end{example}

\begin{table}
  \begin{adjustwidth}{-\tableaddwidth}{-\tableaddwidth}\centering
  \begin{tabular}{R@{$\; M(\,$}R@{, }R@{, }R@{, }R@{, }R@{, \;\;}R@{$\,)$ \quad}RRRRRCC}
    \toprule
             & k_1 & k_2 & k_3  & l_1 & l_2 & l_3 & \multicolumn{1}{C}{\r} & \multicolumn{1}{C}{\quad s} & \multicolumn{1}{C}{\Sigma} & \multicolumn{1}{C}{p_1} & s_{22}                  & s_{2}                              \\
    \midrule
    M_0 :=   & 35  & 21  & -34  & 12  & 10  & 0   & 1\,289                 & 499                         & 1                         & 248                     & \multicolumn{2}{c}{[condition C fails]}                      \\[\pairskip]
    M_1 :=   & 440 & 168 & -320 & 159 & 129 & 0   & 141\,151               & -58\,968                    & 0                         & 42\,822                 & 0                       & \nicefrac{\minus 35047}{141151}    \\
    M_2 :=   & 400 & 168 & -352 & 165 & 51  & 0   & 141\,151               & -58\,968                    & 0                         & 42\,822                 & \multicolumn{2}{c}{[condition C fails]}                      \\[\pairskip]
    M_3 :=   & 410 & 259 & -457 & 192 & 20  & 0   & 203\,383               & -79\,707                    & -1                        & 66\,848                 & \nicefrac{ \minus 1}{6} & \nicefrac{ \plus  614891}{2440596} \\
    M_4 :=   & 548 & 497 & -335 & 374 & 336 & 0   & 203\,383               & -79\,707                    & -1                        & 50\,833                 & \nicefrac{ \minus 1}{6} & \nicefrac{ \minus 621835}{2440596} \\
    M_5 :=   & 370 & 287 & -457 & 126 & 74  & 0   & 203\,383               & -79\,707                    & -1                        & 24\,056                 & \nicefrac{ \minus 1}{6} & \nicefrac{ \plus  404657}{2440596} \\
    M_6 :=   & 610 & 491 & -325 & 462 & 314 & 0   & 203\,383               & -79\,707                    & -1                        & 130\,561                & \nicefrac{ \minus 1}{6} & \nicefrac{ \plus  123017}{2440596} \\
    M_7 :=   & 650 & 491 & -305 & 432 & 404 & 0   & 203\,383               & -79\,707                    & -1                        & 147\,241                & \nicefrac{ \minus 1}{6} & \nicefrac{ \plus  659411}{2440596} \\
    M_8 :=   & 548 & 469 & -355 & 432 & 230 & 0   & 203\,383               & -79\,707                    & -1                        & 76\,945                 & \nicefrac{ \minus 1}{6} & \nicefrac{ \minus 947995}{2440596} \\
    \bottomrule
    \addlinespace
  \end{tabular}
  \end{adjustwidth}
  \caption{Some examples of positively curved Eschenburg spaces.}
  \label{table:examples}
\end{table}


\section{Proof of \cref{THM: nondiffeo souls with positive curv}}\label{sec: Proof of THM A}

We are now ready to prove our main result.  By \cref{THM: existence of tangential nonhomeo pairs}, there exist pairs of positively curved Eschenburg spaces \(M_1\), \(M_2\) that are tangentially homotopy equivalent but non-homeomorphic. Pick one such pair and a tangential homotopy equivalence \(f\colon M_1\to M_2\).  We claim that \(M:=M_2\) has the property stated in \cref{THM: construction for THM A}. Indeed, let $E\to M_2$ be an arbitrary real vector bundle of rank $\geq 8$.  Denote by \(f^*E\to M_1\) its pullback along~\(f\). The induced map $h\colon f^*E\to E$ is still a tangential homotopy equivalence, see for example the proof of Proposition~1.3 in \cite{GoZ}.  Now we need the following well-known corollary of a classical result of Siebenmann; it appears, for example, as Theorem~10.1.6 in \cite{TW}, where it is dubbed ``Work Horse Theorem'':

\begin{theorem}[Siebenmann, Belegradek]\label{WHT}
\emph{Let \(E_1\to M_1\) and \(E_2\to M_2\) be two vector bundles of the same rank \(l\) over two closed manifolds of the same dimension~\(n\). Suppose that \(l\geq 3\) and \(l>n\).  Then any tangential homotopy equivalence \(h\colon  E_1 \to E_2\) is homotopic to a diffeomorphism.}
\end{theorem}
\begin{proof}[Proof sketch]
  Note first that we might as well assume \(M_1\) and \(M_2\) to be connected, as we may argue one component at a time.
  For \(n = 0\) or \(n = 1\), the statement can be checked by elementary means.  For \(n\geq 2\), a proof is outlined in \cite{Bel} below Proposition~5, as follows:
  First one observes that the total space \(E\) of a vector bundle of rank \(\geq 3\) over a closed connected manifold \(M\) of dimension \(\geq 2\) satisfies hypothesis (3) in \cite[Theorem~2.2]{Sie69}: it has one end, \(\pi_1\) is essentially constant at \(\infty\), and \(\pi_1(\infty)\to\pi_1(E)\) is an isomorphism.
  Thus, if such a total space contains an embedded closed connected manifold \(S\) such that the embedding \(S \hookrightarrow E\) is a homotopy equivalence, then \(E\) admits the structure of a vector bundle over \(S\), with the given embedding as zero section.  Slight generalizations of the arguments used in the proof of \cite[Theorem 2.3]{Sie69} then complete the proof:  For \(h\colon E_1\to E_2\) as above and \(s_1\colon M_1\to E_1\) the zero section, the homotopy equivalence \(h\circ  s_1\colon M_1\to E_2\) is homotopic to a smooth embedding \(g\colon M_1\to E_2\) by general position arguments \cite[Ch.\,2, Theorems~2.6 and 2.13]{Hirsch}.  It follows that \(E_2\) has the structure of a vector bundle over \(M_1\) and can be identified with the normal bundle \(N_g\) of the embedding~\(g\).  On the other hand, the assumption that \(h\) is a \emph{tangential} homotopy equivalence implies that the vector bundles \(N_g\) and \(E_1\) over \(M_1\) are stably isomorphic, and since their rank \(l\) is greater than \(n\) it follows that \(N_g\cong M_1\) (see the reference given in the proof of \cref{cor: sum of line bundles}).
\end{proof}

Returning to the proof of \cref{THM: nondiffeo souls with positive curv}, we find that the total spaces of our bundles $f^*E\to M_1$ and $E\to M_2$ are diffeomorphic.
By \cref{THM: bundles over Esch}, they admit two metrics with non-negative sectional curvature, one with soul isometric to $M_1$ and the other with soul isometric to $M_2$.  This completes the proof of \cref{THM: construction for THM A}\slash\cref{THM: nondiffeo souls with positive curv}.

\medskip

The pairs of souls we have constructed have codimension \(\geq 8\).  This is probably not optimal.  All we know is that any pair of souls as in \cref{THM: nondiffeo souls with positive curv} necessarily has codimension at least three:  according to \cite{BKS11}, any two codimension-two souls of a simply connected open manifold are homeomorphic.
There is, however, the following result on positively-curved codimension-two souls due to Belegradek, Kwasik and Schultz:

\begin{theorem}[\cite{BKS15}]\label{THM: BKS15 nontrivial line bundles}
There exist Eschenburg spaces $M$ with the following property:  the total space of every non-trivial complex line bundle over $M$ admits a pair of non-diffeomorphic, homeomorphic souls of positive sectional curvature.
\end{theorem}

Indeed, this is essentially the case \(m=0\) of Theorem~1.4 in \cite{BKS15}; the exact statement may easily be extracted from the proof of this theorem given there (see page~41).  This result does not rely on the ``Work Horse Theorem'' stated as \cref{WHT} above.  Rather, the main topological tool that goes into it is Theorem~12.1 of loc.\ cit.:

\emph{Let \(M_1\), \(M_2\) be two closed simply connected manifolds of dimension $n\geq 5$ with $n\neq 1\mod 4$, such that $M_1$ is the connected sum of $M_2$ with a homotopy sphere that bounds a parallelizable manifold. Let \(L_2\to M_2\) be a non-trivial line bundle, and let $L_1\to M_1$ be its pullback via the standard homeomorphism $M_1\to M_2$. Then the total spaces $L_1$ and $L_2$ are diffeomorphic.}


\section{Moduli spaces of Riemannian metrics}\label{sec: mmoduli spaces}

Given a manifold $N$, denote by $\mathcal{R}(N)$ the space of all (complete) Riemannian metrics on $N$. We refer to \cite[Chapter 1]{TW} for basic properties of spaces of metrics. They can be topologized in different ways.  Following \cite{BKS11}, we consider:
\begin{itemize}
\item[(u)] the topology of uniform $C^\infty$-convergence
\item[(c)] the topology of uniform $C^\infty$-convergence on compact subsets
\end{itemize}
The space of metrics equipped with one of these topologies will be denote \(\mathcal{R}^{u}(N)\) and $\mathcal{R}^{c}(N)$, respectively.  The diffeomorphism group $\text{Diff}(N)$ acts on $\mathcal{R}(N)$ by pulling back metrics. This action is continuous with respect to both topologies.  The quotient spaces are called the \textbf{moduli spaces of metrics} and will be denoted by $\mathcal{M}^c(N)$ and $\mathcal{M}^{u}(N)$, respectively. While $\mathcal{M}^c(N)$ is always path-connected, $\mathcal{M}^{u}(N)$ can have uncountably many connected components if $N$ is non-compact.

For an open manifold $N$, we are interested in the subspace $\mathcal{R}_{K\geq 0}(N)$ of $\mathcal{R}(N)$ consisting of all metrics with non-negative sectional curvature. Pulling back metrics preserves curvature bounds, so we can consider the corresponding moduli spaces  $\mathcal{M}_{K\geq 0}^{u}(N)$ and $\mathcal{M}_{K\geq 0}^{c}(N)$. Connectedness properties of these spaces have been the subject of much research; see \cite{Tu} and \cite[Chapter 10]{TW} for recent surveys on this topic.

Our main result \cref{THM: nondiffeo souls with positive curv} suggests to also consider the subspace of those metrics with non-negative sectional curvature $K\geq 0$ whose souls $S$ have positive sectional curvature $K^S>0$.  We will denote this subspace and the the corresponding moduli space by $\mathcal{R}_{K\geq 0,K^S> 0}(N)$ and $\mathcal{M}_{K\geq 0,K^S> 0}(N)$, with the appropriate superscript again indicating the topology.
Let us examine how the results above are reflected in the connected properties of these subspaces.  We first consider the two topologies separately and then discuss the special case of codimension one souls, for which both topologies coincide.

\subsection*{Topology of uniform convergence}\label{SS: uniform convergence}
The following result is an immediate consequence of Theorem 1.5 in \cite{BKS11}:
\begin{quote}
  \emph{Let $g_1,g_2\in \mathcal{R}_{K\geq 0}^{u}(N)$ with souls $S_1,S_2$. If $S_1,S_2$ are non-diffeomorphic, then the equivalence classes of $g_1,g_2$ lie in different path components of $\mathcal{M}_{K\geq 0}^{u}(N)$.}%
\end{quote}
So $\mathcal{M}_{K\geq 0,K^S> 0}^u(N)$ is not path-connected for any $N$ as in \cref{THM: construction for THM A} or \cref{THM: BKS15 nontrivial line bundles}.

\subsection*{Topology of uniform convergence on compact subsets}\label{SS: uniform convergence on compact subsets}

The following result is an immediate consequence of Lemma 6.1 in \cite{KPT}:
\begin{quote}
  \emph{Let $g_1,g_2\in \mathcal{R}_{K\geq 0}^{c}(N)$ with souls $S_1,S_2$.  Assume that the normal bundles of \(S_1\) and \(S_2\) in \(N\) both have non-trivial rational Euler class. If $S_1,S_2$ are non-diffeomorphic, then the equivalence classes of $g_1,g_2$ lie in different path components of $\mathcal{M}_{K\geq 0}^{c}(N)$.}
\end{quote}
For dimensional reasons, the Euler classes of the spaces in \cref{THM: construction for THM A} vanish. On the other hand, the Euler classes of the spaces in \cref{THM: BKS15 nontrivial line bundles} are non-zero by assumption. Thus,  $\mathcal{M}_{K\geq 0,K^S> 0}^c(N)$ is not path-connected when $N$ is a manifold as in \cref{THM: BKS15 nontrivial line bundles}.

\subsection*{Codimension one souls}\label{SS: codimension one moduli}
In the special case where the souls have codimension one in $N$ both topologies coincide. More precisely, the following result is Proposition 2.8 in \cite{BKS11}:
\begin{quote}
\emph{If $N$ admits a metric with non-negative curvature and codimension-one soul, then the obvious map \(\mathcal{M}_{K\geq 0}^{u}(N)\to\mathcal{M}_{K\geq 0}^{c}(N)\) is a homeomorphism. Moreover, the natural map
\[\textbf{soul: } \mathcal{M}_{K\geq 0}^{u}(N)\to\coprod_i\mathcal{M}_{K\geq 0}(S_i)\]
that assigns to each metric the metric of its soul is a homeomorphism as well, where \(\coprod\) denotes disjoint union over are all possible diffeomorphism types  \(S_i\) for souls of \(N\).}
\end{quote}
When $N$ is simply-connected all codimension-one souls $S$ are diffeomorphic, so that the map \(\textbf{soul: } \mathcal{M}_{K\geq 0}^{u}(N)\to\mathcal{M}_{K\geq 0}(S)\) is a homeomorphism. We can use this result to obtain further open manifolds \(N\) such that $\mathcal{M}_{K\geq 0,K^S> 0}^{u}(N)$ is not path-connected:
Kreck and Stolz showed in \cite{KS93} that there are Eschenburg spaces $M$ for which the moduli space $\mathcal{M}_{K> 0}(M)$ of metrics with \emph{positive} sectional curvature is not path-connected. By considering Riemannian products with the real line we find that $\mathcal{M}_{K\geq 0,K^S> 0}^{u}(M\times\RR)$ is not path-connected either.


\end{document}